\documentclass[12pt]{amsart}

\usepackage{amssymb}
\usepackage{mathtools}

\usepackage{fullpage}
\usepackage{hyperref}

\newtheorem{Theorem}{Theorem}
\newtheorem{Lemma}[Theorem]{Lemma}
\newtheorem{Corollary}[Theorem]{Corollary}
\newtheorem{Proposition}[Theorem]{Proposition}

\theoremstyle{definition}
\newtheorem{Definition}[Theorem]{Definition} 
\newtheorem{Example}[Theorem]{Example}

\theoremstyle{remark}
\newtheorem{Remark}[Theorem]{Remark}

\numberwithin{equation}{section}
\numberwithin{Theorem}{section}

\def\R{{\mathbb R}}
\def\Z{{\mathbb Z}}
\def\C{{\mathbb C}}

\def\cK{{\mathcal K}}

\def\a{{\mathfrak a}}
\def\b{{\mathfrak b}}
\def\g{{\mathfrak g}}
\def\h{{\mathfrak h}}
\def\i{{\mathfrak i}}
\def\n{{\mathfrak n}}

\def\gl{\mathfrak{gl}}
\def\osp{\mathfrak{osp}}
\def\psl{\mathfrak{psl}}
\def\sl{\mathfrak{sl}}

\def\0{{\overline 0}}
\def\1{{\overline 1}}

\def\codim{\mathrm{codim}}
\def\ind{\mathrm{ind}}
\def\Min{\mathrm{Min}}
\def\rank{\mathrm{rank}}
\def\Max{\mathrm{Max}}

\title{On the Index of Borel Subalgebras of Lie Superalgebras}
\author{Simon M.~Goodwin and Samuel Renforth}
\date{\today}

\begin{document}

\begin{abstract}
Let $\b=\h\oplus\n$ be a Borel subalgebra of a basic classical Lie superalgebra over $\C$ with 
$\h$ a Cartan subalgebra. We give an upper bound for the index $\ind(\b)$ of $\b$; in some 
instances this bound is $0$, in which case $\ind(\b)=0$. Additionally we prove that $\ind(\b,\n)=0$, which implies $\ind(\b,\i)=0$ 
for all ideals $\i \subseteq \n$ of $\b$. We also show that $\ind(\b,\a^*)=0$ for all abelian ideals $\a \subseteq \n$ of $\b$. 
These results are achieved by extending the theory of strongly orthogonal roots and the Kostant cascade to the theory of 
generalized root systems developed by Dimitrov and Fioresi.
\end{abstract}

\maketitle

\section{Introduction}
In the 1970s, the {\em index} of a Lie algebra was defined by Dixmier in~\cite{Di}. Let $\g$ be a Lie algebra over $\C$, then 
its index is 
\begin{equation} \label{e:index}
\ind(\g) \coloneqq \min_{\xi \in \g^*}\,\dim(\g_{\xi}),
\end{equation}
where $\g_{\xi}\coloneqq\{x\in\g:\xi([x,y])=0 \text{ for all } y \in \g \}$ is the centraliser in $\g$ of $\xi \in \g^*$ for the coadjoint action. 
There has been great interest in the index of Lie algebras due to its relevance in invariant
theory and in representation theory, for example from A. Joseph, \cite{Jo2}, and D. Panyushev, \cite{Pa1}. 
We refer to the introduction of 
\cite{CMo} and the references therein for an overview of results
about the index of Lie algebras.

The index of a Lie algebra appears to be difficult to calculate in general, although substantial progress 
has been made in specific cases.
If $\g$ is reductive, then it well known that $\ind(\g)$ equals 
$\rank(\g)$. The index of Borel subalgebras is known and was first established in \cite[Theorem 4.1]{Tr} 
with another approach given in~\cite[Th\'eor\`eme 2.9]{TY}. 
Subsequently, the index of parabolic and more generally seaweed subalgebras has been studied, see for example~\cite{Jo2} 
and~\cite{CCH}.

In this paper we initiate the study of the index of a Lie superalgebra $\g$ over $\C$ by considering the definition of the 
index as given by \eqref{e:index}. 
It can be expected that the index is an important invariant of Lie superalgebras so it is of interest to 
determine in key examples. 

In the case that $\g$ is a basic classical Lie superalgebra it is straightforward to show that $\ind(\g) \le \rank(\g)$ by considering regular semisimple elements in $\g_\0$. The proof of equality from the non-super case 
does not extend easily, though it may be expected that there is equality.

Our main result concerns the index of Borel subalgebras of basic classical Lie superalgebras, and we provide an upper bound for this index. We state this in the 
below theorem and explain the notation and terminology used later in the introduction.

\begin{Theorem} \label{T:index0}
Let $\g$ be $\gl(m|n)$ or a basic classical Lie superalgebra excluding $\psl(2|2)$. 
Let $\b$ be a Borel subalgebra of $\g$ with the associated set of positive roots $\Delta^+$. 
Then $\ind(\b)\le \rank(\g)-|\cK(\Delta^+)|$ where $\cK(\Delta^+)$ is the set of
roots in the Kostant cascade of $\Delta^+$.
\end{Theorem}

The key to proving this theorem is extending the theory of strongly orthogonal roots and the Kostant cascade to the theory of 
generalized root systems (GRSs) developed by Dimitrov and Fioresi in~\cite{DF}.
In Theorem~4.7 of {\em loc.\ cit.\ }it is proved that basic classical Lie superalgebra root systems are 
examples of GRSs. In turn this allows the proof of \cite[Th\'eor\'eme 2.9]{TY}
to be extended to Lie superalgebras. The statement of Theorem~\ref{T:index0}
is the same as the corresponding statement for simple Lie algebras except that it is an inequality 
rather than equality. It may be expected that this inequality is actually an equality, though
the methods to prove this in the non-super case do not extend easily. As in the case for simple Lie algebras
it is sometimes the case that $|\cK(\Delta^+)| = \rank(\g)$, and in those cases we do have the 
equality $\ind(\b) = 0$; an example being $\g = \osp(m,2n)$ with $m \not\equiv 2 \bmod 4$ and 
the standard choice of $\b$.  We refer to Remark~\ref{R:psl22} 
for an explanation of why $\psl(2|2)$ is excluded, where we also note that this case
should be accessible through computational methods.

The {\em Kostant cascade}, originally defined by Kostant in~\cite{Ko} and named by 
Panyushev in~\cite{Pa2}, is a subset of a Lie algebra root system $\Delta$ determined 
by the algorithm outlined below, with more detail given in \S\ref{ss:sorkc}. 
The algorithm produces a nested sequence of proper subsystems
\begin{equation} \label{e:Deltai}
\Delta=\Delta_0\supsetneq\Delta_1\supsetneq\cdots\supsetneq
\Delta_N\supsetneq\Delta_{N+1}=\varnothing,
\end{equation}
where $\dim(\R\Delta_{i+1}) < \dim(\R\Delta_{i})$. The sequence is constructed by starting with a set of positive roots $\Delta^+$ of $\Delta$, setting
$\Delta_0^+ = \Delta^+$, then proceeding inductively as follows:
\begin{enumerate}
\item[(KC1)] find the highest roots $\Max(\Delta^+_i)$ of $\Delta^+_i$;
\item[(KC2)] construct the positive system $\Delta^+_{i+1}$ of $\Delta_{i+1}$ by removing all roots of the form $\theta-\delta\in\Delta^+\setminus \{0\}$ from $\Delta^+_i$, where $\theta \in \Max(\Delta^+_i) \cup \{0\}$ and $\delta\in\Delta^+_i$;
\item[(KC3)] repeat until $\Delta^+_{N+1} = \varnothing$.
\end{enumerate}
It is straightforward to verify that 
$\Delta^+_{i+1} = \{\alpha \in \Delta_i^+ : \langle \alpha,\theta \rangle=0 \text{ for all } \theta \in \Max(\Delta_i^+)\}$ given the properties of the standard inner product $\langle\_\,,\_\rangle$ on $\R\Delta$, and from this inductively that
each $\Delta_i$ is a root system.
One then takes 
\begin{equation} \label{e:KC}
\cK(\Delta^+)\coloneqq\bigsqcup_{i=0}^N\Max(\Delta^+_i)
\end{equation}
to be the roots in the Kostant cascade of $\Delta^+$. 

The Kostant cascade can likewise be defined for GRSs using the exact same algorithm.
In \S\ref{ss:sorkc} we observe that the procedure for defining the Kostant cascade is valid for GRSs.
Furthermore, we show that the output, $\cK(\Delta^+)$, is 
a strongly orthogonal subset of $\Delta$, which, in particular, implies $|\cK(\Delta^+)| \le \dim(\R\Delta)$.

Now let $\g$ be a Lie superalgebra and let $V$ be a $\g$-module. We can extend the index 
to $\g$-modules, see for example~\cite{Pa1}, by defining
\begin{equation}\label{e:index2}
\ind(\g,V)\coloneqq\underset{\xi\in V^*}{\min}\,\codim_{V^*}(\g\cdot\xi)=\dim(V)-\underset{\xi\in V^*}{\max}\,\dim(\g\cdot \xi).
\end{equation}
Note $\ind(\g,\g)=\ind(\g)$ if $\g$ is viewed as a $\g$-module under the adjoint action.

In~\cite{Jo1} it is established that if $B$ is a Borel subgroup of a reductive algebraic group 
$G$ over an algebraically closed field, and if $\b = \h \oplus \n$ is the Lie algebra of $B$, then $B$ has a dense orbit in $\n^*$. 
This implies that $\ind(\b,\n)=0$, and more generally that $\ind(\b,\i)=0$ for all ideals $\i \subseteq \n$ of $\b$.  

The following theorem is a consequence of the proof of Theorem~\ref{T:index0}, extending the 
results for $\ind(\b,\n)$ and $\ind(\b,\i)$ to basic classical Lie superalgebras.

\begin{Corollary}\label{C:indn}
Let $\g$ be $\gl(m|n)$ or a basic classical Lie superalgebra excluding $\psl(2|2)$. Let $\b=\h\oplus\n$ be a Borel subalgebra of $\g$. Then $\ind(\b,\n)=0$, and in general $\ind(\b,\i)=0$ for all ideals $\i \subseteq \n$ of $\b$.
\end{Corollary}

Returning to Lie algebras, it is in general the case that $\ind(\b,\a^*)\ne 0$ for $\a \subseteq \n$ an ideal of $\b$, although it was shown by Panyushev and R{\"o}hrle that $\a^*$ has finitely many $B$-orbits when $\a$ is abelian by~\cite[Corollary 2.4]{PR}. This was later strengthened by Panyushev with~\cite[Theorem 3.2]{Pa3}, which gives a bijection between the $B$-orbits of $\a$ and the subsets of strongly orthogonal roots in $\a$. The finiteness of $B$-orbits implies that $\ind(\b,\a^*)=0$ for $\a  \subseteq  \n$ an abelian ideal of $\b$.

We extend the result for $\ind(\b,\a^*)$ to Lie superalgebras when $\a \subseteq \n$ is an abelian ideal of $\b$. This is done by employing a comparable algorithm to that for the Kostant cascade where we add instead of subtract roots, relying on Panyushev's observations on strongly orthogonal roots in~\cite{Pa3} applied to GRSs.

\begin{Theorem}\label{T:inda}
Let $\g$ be $\gl(m|n)$ or a basic classical Lie superalgebra excluding $\psl(2|2)$. Let $\b=\h\oplus\n$ be a Borel subalgebra of $\g$. Then $\ind(\b,\a^*)=0$ for all abelian ideals $\a \subseteq \n$ of $\b$.
\end{Theorem}

We outline the contents of this paper. In Section \ref{S:abelianideals} we give a brief overview of basic classical Lie superalgebras and the (abelian) ideals of their Borel subalgebras. In Section \ref{S:GRS} we review the theory of generalized root systems and provide the definitions and results that we require, noting in particular that Lie (super)algebra root systems are GRSs. In \S\ref{ss:sorkc} we define subsets of strongly orthogonal roots and extend the Kostant cascade to GRSs. In \S\ref{ss:abelianorthog} we find connections between abelian ideals of GRSs and strongly orthogonal roots. In Section \ref{s:evalindex} we use our findings to prove Theorem~\ref{T:index0}, Corollary~\ref{C:indn}, and Theorem~\ref{T:inda}.

\section{Lie Superalgebras and Ideals} \label{S:abelianideals}

\subsection{Basic classical Lie superalgebras}
In the 1970s, the finite-dimensional simple Lie superalgebras $\g$ over $\C$ were classified by Kac in~\cite{Ka}. We refer the reader 
to~\cite{CW} or \cite{Mu} as general references for the theory of Lie superalgebras used in this paper. 
Amongst the simple Lie superalgebras are the {\em basic classical} Lie superalgebras, which are those where $\g_\0$ is a reductive Lie algebra, and $\g$ admits a non-degenerate even invariant bilinear form $(\_\,,\_)$. These simple Lie superalgebras have the most in common with simple Lie algebras due to the Killing form-esque non-degeneracy and invariance of $(\_\,,\_)$. The basic classical Lie superalgebras that are not Lie algebras are classified as follows:
\begin{equation*}
\begin{matrix}
\sl(m|n)&\psl(n|n)&\osp(M|2n)&\text{D}(2,1;\alpha)&\text{F}(4)&\text{G}(3).
\\
m > n\ge 1& n\ge 2& M,n\ge 1&\alpha\in\C\setminus\{0,-1\}&&
\end{matrix}
\end{equation*}

Let $\g = \g_\0 \oplus \g_\1$ be a basic classical Lie superalgebra, let $\h$ be a 
Cartan subalgebra of $\g_\0$, and let $\rank(\g)\coloneqq \dim(\h)$ be the rank of $\g$. 
Then $\g$ has the root spaces $\g_{\alpha}\coloneqq\{x\in\g \mid [h,x] = \alpha(h)x \text{ for all h} \in \h\}$ for 
all $\alpha\in\h^*$, root system $\Delta=\{\alpha\in\h^*\setminus\{0\}:\g_{\alpha}\ne 0\}$, and 
root space decomposition $\g=\h\oplus\bigoplus_{\alpha\in\Delta}\g_{\alpha}$. 
A subset $\Delta^+ \subseteq \Delta$ is a set of positive roots of $\Delta$ if
exactly one of $\alpha$ and $-\alpha$ is an element of $\Delta^+$ for all $\alpha\in\Delta$, and
if $\alpha,\beta\in\Delta^+$ and $\alpha+\beta\in\Delta$, then $\alpha+\beta\in\Delta^+$.
The subalgebra associated with $\Delta^+$ is then denoted $\n=\bigoplus_{\alpha\in\Delta^+}\g_{\alpha}$, and 
the corresponding Borel subalgebra of $\g$ is given by $\b=\h\oplus\n$.

\begin{Remark}
The general linear Lie superalgebra $\gl(m|n)$ is not simple but obeys the other two conditions
for being basic classical. As such, $\gl(m|n)$ has many properties of a basic classical Lie superalgebra, 
so the techniques of this paper also apply to this case. Throughout the paper we refer only to 
basic classical Lie algebras but emphasise that all results apply also in the case $\g =\gl(m|n)$.
\end{Remark}

\begin{Remark} \label{R:psl22}
Of the basic classical Lie superalgebras, $\psl(2|2)$ is the only one for which not all root spaces are 
1-dimensional, as it has some 2-dimensional root 
spaces. For this reason some of our arguments do not apply to $\psl(2|2)$ and we thus omit $\psl(2|2)$ in some parts of the paper and it is excluded from the statement of our main results. Given the size of $\psl(2|2)$, we expect that it can be dealt with by explicit calculations.
We implicitly do not include this case throughout the remainder of this paper to avoid dealing with the technical issues around this.
\end{Remark}

\subsection{(Abelian) Ideals of Borel Subalgebras}
We continue to let $\g$ be a basic classical Lie superalgebra, where we implicitly assume $\g \ne \psl(2|2)$ as 
explained in Remark~\ref{R:psl22}.

Recall that a $\Z_2$-graded subspace $\i  \subseteq  \n$ is an {\em ideal} of $\b$ if $[\b,\i] \subseteq \i$, 
and a $\Z_2$-graded subspace $\a \subseteq \n$ is an {\em abelian ideal} of $\b$ if $[\b,\a] \subseteq \a$ 
and $[\a,\a]=0$.

Let $\i \subseteq \n$ be an ideal of $\b$. As $\h$ acts diagonally on $\g$, 
it follows that there exists a subset $\Delta^+_{\i}  \subseteq  \Delta^+$ such that
\begin{equation} \label{e:idecomp}
\i=\bigoplus_{\alpha\in\Delta^+_{\i}}\g_{\alpha}.
\end{equation}
Note that $[\g_{\alpha},\g_{\beta}]=\g_{\alpha+\beta}$ for all $\alpha,\beta\in\Delta$ with 
$\alpha\ne-\beta$, see for example \cite[Lemma 2.1.1]{Mu}. In combination with \eqref{e:idecomp}, 
this allows us to translate $\i$ being an ideal into the language of roots:
\begin{align*} 
[\b,\i] \subseteq  \i&&\text{implies}
&&\bigoplus_{\substack{\alpha\in\Delta^+ \\ \beta\in\Delta^+_{\i}}}[\g_{\alpha},\g_{\beta}] \subseteq  \i
&&\text{implies}
&&\bigoplus_{\substack{\alpha\in\Delta^+ \\ \beta\in\Delta^+_{\i}}}\g_{\alpha+\beta} \subseteq 
\bigoplus_{\gamma\in\Delta_{\i}^+}\g_{\gamma}.
\end{align*}
If $\a \subseteq \n$ is an abelian ideal of $\b$, then we further find
\begin{align*}
[\a,\a]=0
&&\text{implies}&&\bigoplus_{\alpha,\beta\in\Delta^+_{\a}}[\g_{\alpha},\g_{\beta}]=0
&&\text{implies}&&\bigoplus_{\alpha,\beta\in\Delta^+_{\a}}\g_{\alpha+\beta}=0.
\end{align*}
This motivates the following definition for subsets of $\Delta^+$.

\begin{Definition} \label{D:abelianideals}
Let $I, A \subseteq  \Delta^+$.
\begin{enumerate}
\item[(i)] We say that $I$ is an {\em ideal} of $\Delta^+$ if $\alpha\in\Delta^+$, $\beta\in I$, 
and $\alpha+\beta\in\Delta^+$ implies $\alpha+\beta\in I$.
\item[(ii)] We say that $A$ is an {\em abelian ideal} of $\Delta^+$ if $A$ 
is an ideal of $\Delta^+$ and if $\alpha,\beta\in A$ implies $\alpha+\beta\notin\Delta$.
\end{enumerate}
\end{Definition}

The above arguments show that the map $\i\mapsto\Delta^+_{\i}$ is a bijection 
from the ideals $\i \subseteq  \n$ of $\b$ to the ideals of $\Delta^+$, and the map 
$\a\mapsto\Delta^+_{\a}$ is a bijection from the the abelian ideals $\a \subseteq  \n$ 
of $\b$ to the abelian ideals of $\Delta^+$.

Note that $\i$ being an ideal of $\b$ is equivalent to $\i$ being a $\b$-module 
under the adjoint action.
Thus translating the definition of index 
from \eqref{e:index2} to the $\b$-module $\i^*$ we obtain
\begin{equation*}
\ind(\b,\i^*)=\dim(\i)-\underset{x\in \i}{\max}\,\dim([\b, x]).
\end{equation*}
Since $(\_\,,\_)$ induces a $\g$-module isomorphism $\g\cong\g^*$, we obtain 
$\i^*\cong\g/\i^{\perp}$ as a $\b$-module, where $\i^{\perp}$ is the subspace of $\g$ 
orthogonal to  $\i$ in with respect to $(\_\,,\_)$. Also note that if $\alpha,\beta\in\Delta$, 
then $(\g_{\alpha},\g_{\beta})=0$ unless $\alpha+\beta=0$, 
and $(\h,\g_{\alpha})=0$. It follows that $\n^{\perp}=\b$ and $\b^{\perp}=\n$, so $\n^*\cong\g/\b$ and $\b^*\cong\g/\n$ 
as $\b$-modules. It then also 
follows from~\eqref{e:idecomp} that if $\i \subseteq \n$, then 
$\i^{\perp}=\b\oplus\bigoplus_{\alpha\in\Delta \setminus \Delta^+_{\i}}\g_{-\alpha}$, so
as a vector space we have $\i^*\cong \bigoplus_{\alpha\in\Delta^+_{\i}}\g_{-\alpha}$. 
Defining $[\b,x]_{\i}\coloneqq ([\b,x]+\i^{\perp})/\i^{\perp} \subseteq \g/\i^{\perp}$ for all $x\in\g$, we then similarly find
\begin{equation*}
\ind(\b,\i)=\dim(\i)-\underset{x\in \g}{\max}\,\dim([\b, x]_{\i}).
\end{equation*}

\section{Generalized Root Systems} \label{S:GRS}

\subsection{Definitions and Preliminary Results}

In 2024, Dimitrov and Fioresi introduced the theory of generalized root systems (GRSs) in~\cite{DF}. We give an overview of GRSs and use 
{\em loc.\ cit.\ }as our general reference for all below.

\begin{Definition}
Let $V$ be a finite-dimensional Euclidean space with inner product $\langle\_\,,\_\rangle$ and let $\Delta \subseteq  V$ be 
a non-empty finite subset of $V$. Then $(\Delta,V)$ is a {\em generalized root system} if $V=\R \Delta$, and 
if the following holds for all $\alpha,\beta\in \Delta$:
\begin{enumerate}
\item[(GRS1)] If $\langle\alpha,\beta\rangle>0$, then $\alpha-\beta\in \Delta$;\label{enum:GRS>0}
\item[(GRS2)] If $\langle\alpha,\beta\rangle<0$, then $\alpha+\beta\in \Delta$;\label{enum:GRS<0}
\item[(GRS3)] If $\langle\alpha,\beta\rangle=0$, then $\alpha+\beta\in \Delta$ if and only if $\alpha-\beta\in \Delta$.\label{enum:GRS=0}
\end{enumerate}
\end{Definition}

The {\em rank} of $(\Delta,V)$ is defined as the dimension of $V$. We tend to omit the vector space and inner 
product and just write $\Delta$ as the GRS. Note that $\langle\alpha,\alpha\rangle>0$ for all nonzero $\alpha\in\Delta$, 
so $0\in \Delta$. Then $\langle \alpha, 0 \rangle = 0$ and 
$0+\alpha \in \Delta$, so $-\alpha \in \Delta$. We call $\Delta=\{0\}$ the {\em trivial} GRS.

We can view Lie algebra root systems as GRSs, with the minor caveat that these 
root systems do not contain $0$ whereas
GRSs do contain 0 by convention. 
It follows from~\cite[Theorem 4.7]{DF} that the root systems of
the basic classical Lie superalgebras can also be viewed as GRSs, as we explain next. 
Let $R$ be the root system of a basic classical Lie superalgebra $\g$ and assume 
$\g \ne \psl(n|n)$. Then $\Delta$ can be 
viewed as subset of a real form $\h_\R$ of the Cartan subalgebra of $\g_\0$. The proof of \cite[Theorem 4.7]{DF} 
implies that there is a vector space isomorphism $\h_\R \to V$ to an inner product space $V$ such 
that the image of $R$ in
$V$ is a GRS. Note that the inner product $\langle\_\,,\_\rangle$ on $V$ does not 
correspond to the restriction of $(\_\,,\_)$
to $\h_\R$ as the latter is not positive definite in general; as this does not impact on the 
methods and results of this paper we do not consider
it further here. In the case $\g = \psl(n|n)$ we need to consider $\Delta$ to be 
in the $2n-1$ dimensional subspace of the space
diagonal matrices in $\gl(n|n)$ spanned by the elements of $\Delta$.

We remark that GRSs were classified in 2024 by Cuntz and M{\"u}hlherr in~\cite{CMu}; they found 
that each GRS of rank at least 2 is the quotient of a Lie algebra root system, although we will not 
employ this result. This extended the classification of rank 2 GRSs by Dimitrov and 
Fioresi from~\cite[Section 5]{DF}.

Much of the theory of root systems extends naturally to GRSs. This allows for the 
definition of irreducible GRSs, and then it can be shown 
every GRS can be uniquely decomposed as a direct sum of irreducible subsystems, see~\cite[Corollary 1.17]{DF}. 
Such subsystems must necessarily be mutually orthogonal. Furthermore, every GRS admits a base 
by~\cite[Proposition 1.6]{DF}, where a base $\Pi$ of $\Delta$ is a subset such that every $\alpha\in \Delta$ 
can be written uniquely in the form $\alpha=\sum_{\beta\in \Pi}k_{\beta}\beta$, 
with either $k_{\beta}\in\Z_{\geq 0}$ for all $\beta\in \Pi$, or $k_{\beta}\in\Z_{\leq 0}$ for all $\beta\in \Pi$.
We may then define the set of positive roots $\Delta^+$ of $\Delta$ with respect to a base, which 
we call a {\em positive system} of $\Delta$. As is shown in \cite[Propositions 1.6 and 2.4]{DF}, a 
positive system is equivalently a subset $\Delta^+$ of $\Delta$ such that $0\in\Delta^+$, 
exactly one of $\alpha$ and $-\alpha$ is an element of $\Delta^+$ for 
all $\alpha\in\Delta\setminus\{0\}$, and if $\alpha,\beta\in\Delta^+$ and 
$\alpha+\beta\in\Delta$, then $\alpha+\beta\in\Delta^+$.

We now list some properties of GRSs that we require later in this paper.
\begin{enumerate}
\item[(P1)] For $\alpha,\beta\in \Delta$, the $\alpha$-string,  
$\{\beta+k\alpha\in \Delta \mid k\in\Z\}$  through $\beta$ is continuous by \cite[Corollary 1.13]{DF}.
\item[(P2)] Each positive system $\Delta^+$ of $\Delta$ induces a partial ordering $\preccurlyeq$ on 
$\Delta$, where $\alpha\preccurlyeq \beta$ if $\beta-\alpha$ can be written as a non-negative 
integral linear combination of elements in $\Delta^+$.
\item[(P3)]  It is a consequence of~\cite[Proposition 1.10]{DF} that if $\Pi$ is a base of $\Delta$ 
with the associated positive system $\Delta^+$, and if $\beta\in\Delta^+\setminus\Pi$, 
then $\beta-\alpha\in\Delta^+$ for some $\alpha\in\Pi$.
\item[(P4)]  If $\Delta$ is irreducible, then $\Delta^+$ has a unique highest root by~\cite[Proposition 1.14]{DF}.
\end{enumerate}

\subsection{(Abelian) Ideals of GRSs}
We define (abelian) ideals of positive systems of GRSs as we did in Definition~\ref{D:abelianideals}. 

There are many results for Lie algebra root systems that are solely reliant on its inherent GRS structure. 
An example of which is~\cite[Lemma 1]{Su}, due to Kostant, which states that the ideals of positive systems 
are uniquely defined by the sum of their roots; the proof only requires a GRS compatible inner product so immediately 
extends to GRSs in general. We state this extended result below.

\begin{Lemma}[Kostant]
Let $\Delta^+$ be a positive system of a GRS. Let $I,J \subseteq  \Delta^+$ be ideals of $\Delta^+$ such that 
$\sum_{\alpha\in I}\alpha=\sum_{\alpha\in J}\alpha$. Then $I=J$.
\end{Lemma}

Restricting to Lie superalgebras then gives the following Corollary

\begin{Corollary}
Let $\Delta^+$ be a set of positive roots of a basic classical Lie superalgebra root system. 
Let $I,J \subseteq \Delta^+$ be ideals of $\Delta^+$ such that $\sum_{\alpha\in I}\alpha=\sum_{\alpha\in J}\alpha$. Then $I=J$.
\end{Corollary}

This will not be employed in the rest of this paper, and is used only to demonstrate the method of 
proof for our main results: prove a result for GRSs, which implies a result for Lie superalgebra 
root systems that uncovers Lie superalgebra structures.

\subsection{Strongly Orthogonal Roots and The Kostant Cascade}  \label{ss:sorkc}

We first define strongly orthogonal roots of GRSs, which extends the definition for Lie algebra root systems in~\cite[Definition 1]{Pa3}.

\begin{Definition}
Let $\Delta$ be a GRS. We say that $\alpha,\beta\in \Delta$ are {\em strongly orthogonal} if 
$\alpha \pm \beta \notin \Delta\setminus\{0\}$. A subset $S \subseteq  \Delta$ is a {\em strongly orthogonal subset} 
if the elements of $S$ are pairwise strongly orthogonal.
\end{Definition}

Note that if $\alpha,\beta\in \Delta$ are strongly orthogonal, then $\langle\alpha,\beta\rangle=0$ by (GRS3).

For the definition of the Kostant cascade we need the set of maximal roots in $\Delta^+$ which is
\begin{equation*}
\Max(\Delta^+)\coloneqq\{\alpha\in \Delta^+:\alpha\nprec\beta\text{ for all }\beta\in \Delta^+\},
\end{equation*}
We note that $\Max(\Delta^+)$ is a strongly orthogonal subset of $\Delta$, because, for $\alpha,\beta\in\Max(\Delta^+)$, we have 
$\alpha+\beta \in \Delta^+$ implies $\alpha,\beta \not\in \Max(\Delta^+)$ and $\alpha-\beta \in \Delta^+$ implies $\beta \not\in \Max(\Delta^+)$.
Note also if $\Delta$ is irreducible, then $\Max(\Delta^+)$ is a singleton set that contains the unique highest root $\theta$ of $\Delta^+$ by (P4).

The definition of the Kostant cascade as recalled in the introduction can be extended to GRSs, with  $\Delta_i$ as defined in \eqref{e:Deltai} 
and $\cK(\Delta^+)$ as defined in \eqref{e:KC}. The argument outlined in the introduction that shows each $\Delta_i$ is a root system, works
equally well in the GRS regime. Note the algorithm terminates at $\Delta^+_{N+1}=\{0\}$ instead of $\varnothing$ for GRSs.

We illustrate this construction for the Lie algebra root system of type $\mathrm A_4$. 

\begin{Example} \label{E:KCA4}
Let $\Delta=\{\epsilon_i-\epsilon_j:1\leq i, j\leq 5\}$, and use the notation $ij\coloneqq\epsilon_i-\epsilon_j$. 
Let $\Delta^+=\Delta^+_0=\{ij:1\leq i,j\leq 5\}$. We construct $\cK(\Delta^+)$ diagrammatically.

\bigskip

\begin{center}
\includegraphics[width=0.9\linewidth]{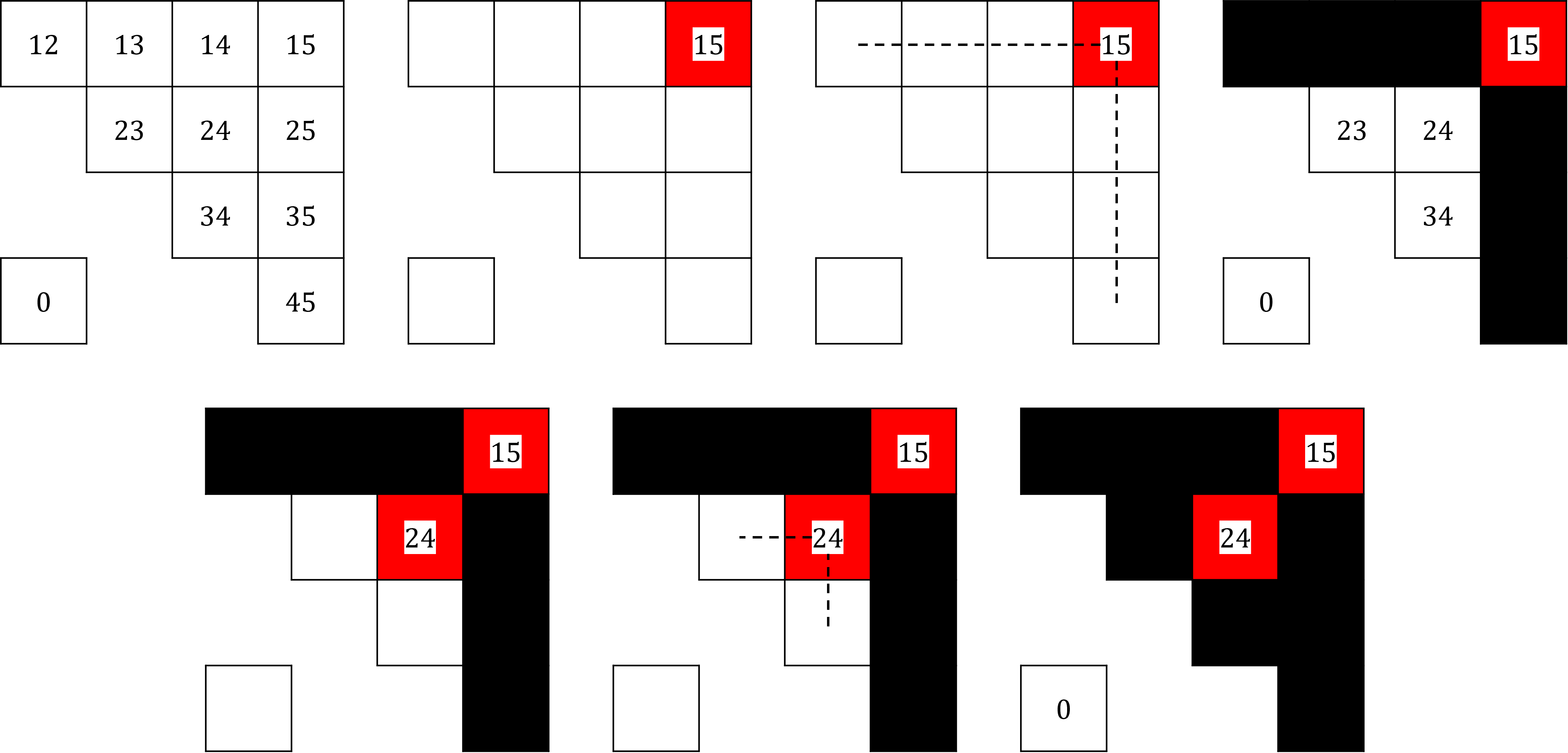}
\end{center}

The construction proceeds by finding $\Max(\Delta^+_0)=\{15\}$, and hence $\Delta_1^+=\{23,24,34,0\}$. 
Then $\Max(\Delta^+_1)=\{24\}$, so we terminate at $N=1$ with $\cK(\Delta^+)=\{15,24\}$.

One may verify the following in this case: $\cK(\Delta^+)$ is a strongly orthogonal subset 
of $\Delta$; every $\alpha\in\Delta^+$ can uniquely be written in the form $\alpha=\theta-\delta$ for 
some $\theta\in\cK(\Delta^+)$ and $\delta\in\Delta^+$; if $24-\alpha\in\Delta^+$ for some 
$\alpha\in\Delta^+$, then $\alpha\in\Delta_1^+$. 
\end{Example}

We formalise the observations of this example in the next theorem. We first give some notation that will be 
used: let $V_i \coloneqq \R\Delta_i$; for each $\theta\in\cK(\Delta^+)$, let $i_{\theta}\in\Z_{\geq 0}$ be such 
that $\theta\in\Max(\Delta^+_{i_{\theta}})$; let 
$\Delta_{i}^{(-)}{}_{\theta}\coloneqq\{\gamma\in \Delta_{i}^+\setminus\{\theta\}:\theta-\gamma\in \Delta_{i}^+\}$ 
for all $\theta\in\cK(\Delta^+)$ and $i=0,\ldots,N$; let $\Delta^{(-)}{}_{\theta}\coloneqq\Delta_{0}^{(-)}{}_{\theta}$.

\begin{Theorem} \label{T:KCstrongorth}
Let $(\Delta,V)$ be a non-trivial irreducible GRS and let $\Delta^+$ be a positive system of $\Delta$. Then the following holds:
\begin{enumerate}
\item[(i)] $\cK(\Delta^+)$ is a strongly orthogonal subset of $\Delta$;
\item[(ii)] $\vert\cK(\Delta^+)\vert\leq\dim(V)$;
\item[(iii)] $\Delta^{(-)}{}_{\theta}=\Delta^{(-)}_{i_{\theta}}{}_{\theta}$ for all $\theta\in\cK(\Delta^+)$;
\item[(iv)] $\Delta^+=\bigsqcup_{\theta\in \cK(\Delta^+)}(\theta-\Delta^{(-)}{}_{\theta})$;\l
\item[(v)] $\Delta^{(-)}{}_{\theta}\cap\Delta^{(-)}{}_{\phi}=\{0\}$ for all distinct $\theta,\phi\in\cK(\Delta^+)$.
\end{enumerate}
\end{Theorem}

\begin{proof}
Let $\cK=\cK(\Delta^+)$.

\noindent
(i) This follows as $\Max(\Delta^+_i)$ is a strongly orthogonal subset for 
$i=0,\ldots,N$ as shown above, and from the definitions we have that all 
elements of $\Delta^+_{i+1}$ are strongly orthogonal to all elements of $\Max(\Delta_j)$ for $j \le i$. 

\noindent
(ii) By (i)we have that $\cK$ is linearly independent, and hence 
$|\cK| \leq \dim(V)$.

\noindent
(iii) We have that $\Delta^{(-)}_{i_{\theta}}{}_{\theta} \subseteq \Delta^{(-)}{}_{\theta}$ 
and $\Delta^{(-)}_j{}_{\theta}  = \varnothing$ for $j > i_\theta$.
So suppose 
that $\theta-\delta\in\Delta^+$ but $\delta\in\Delta^+ \setminus \Delta^+_{i_{\theta}}$. It follows 
that $\phi-\delta\in\Delta^+_{i_{\phi}}$ for some $\phi\in\cK$ with $i_{\phi}<i_{\theta}$, and hence 
$\delta\in\Delta^+_{i_{\phi}}$. We have $\langle\phi,\delta\rangle > 0$, otherwise 
$\phi \not\in\Max(\Delta^+_{i_{\phi}})$. But then $\langle\theta-\delta,\phi\rangle=-\langle\delta,\phi\rangle<0$, 
so $\phi+(\theta-\delta)\in\Delta^+$, which contradicts $\phi\in\Max(\Delta^+_{i_{\phi}})$. Hence, no 
such $\delta$ exists so $\Delta^{(-)}{}_{\theta} = \Delta^{(-)}_{i_{\theta}}{}_{\theta}$.

\noindent
(iv) This follows from (iii) and the definition of $\Delta^+_i$ and $\cK$.

\noindent
(v) Assume for a contradiction that there exists $0\ne\gamma\in\Delta^{(-)}_{i_{\theta}}{}_{\theta}\cap\Delta^{(-)}_{i_{\phi}}{}_{\phi}$ 
for some distinct $\theta,\phi\in\cK$ with
$i_{\phi} \le i_{\theta}$. 
As $\gamma = \phi-(\phi-\gamma)\in(\phi-\Delta^{(-)}_{i_{\phi}}{}_{\phi})$, we get that $\gamma\notin\Delta^+_j$ for $j >i_\phi$, 
so we have $i_{\phi} = i_{\theta}$.
Then since $\theta$ and $\phi$ live in distinct orthogonal irreducible components of $\Delta_{i_\theta}$ by (P4).
But this implies $\gamma$ is in both of these irreducible components which is the required contradiction.
\end{proof}

It is a consequence of Theorem~\ref{T:KCstrongorth}(iv) and (P4) that every 
$\alpha\in\Delta^+$ can uniquely be written in the form $\alpha=\theta-\delta$ for some 
$\theta\in\cK(\Delta^+)$ and $\delta\in\Delta^+$. 

Of course $|\cK(\Delta^+)|$ is independent of the choice of positive roots $\Delta^+$ of a Lie algebra root system. 
The below example shows that $|\cK(\Delta^+)|$ depends on the positive system of a GRS $\Delta$ in general.

\begin{Example}
Consider the basic classical Lie superalgebra $\osp(2|6)$ with Cartan subalgebra $\h$. Following~\cite[\S 1.2.5]{CW}, there exists a base $\{\epsilon,\delta_1,\delta_2,\delta_3\}$ of $\h^*$ such that
\begin{equation*}
\Delta=\{\pm\delta_i\pm\delta_j:1\leq i\ne j\leq 3\}\cup\{\pm\delta_i\pm\epsilon:1\leq i\leq 3\}.
\end{equation*}
The below table describes the roots of the Kostant cascade of two positive systems $\Delta^+$ and $\Omega^+$ 
of $\Delta$. We partition $\Delta^+$ into the subsets $\theta-\Delta^{(-)}{}_{\theta}$ for $\theta\in\cK(\Delta^+)$, 
and $\Omega^+$ into the subsets $\theta-\Omega^{(-)}{}_{\theta}$ for $\theta\in\cK(\Omega^+)$. We 
find that $\cK(\Delta^+)=\{\delta_1+\epsilon,2\delta_2\}$ and $\cK(\Omega^+)=\{2\delta_1,\delta_2+\epsilon,2\delta_3\}$.

\begin{table}[!ht]
\centering
\begin{tabular}{c|cccccc|}
\cline{2-7}
  $(\delta_1+\epsilon)-\Delta^{(-)}_{\delta_1+\epsilon}$ & ${\epsilon-\delta_1}$ & ${\epsilon-\delta_2}$ & ${\epsilon-\delta_3}$ & ${\delta_3+\epsilon}$ & ${\delta_2+\epsilon}$ & $\boxed{\delta_1+\epsilon}$ \rule{0pt}{2.6ex}\\
&& ${\delta_1-\delta_2}$ & ${\delta_1-\delta_3}$ & ${\delta_3+\delta_1}$ & ${\delta_2+\delta_1}$ &${2\delta_1}$\\
\cline{2-7}
  $(2\delta_2)-\Delta^{(-)}_{2\delta_2}$&&& ${\delta_2-\delta_3}$& ${\delta_3+\delta_2}$ & $\boxed{2\delta_2}$ &\\
&&&& ${2\delta_3}$ &&\rule[-0.9ex]{0pt}{0pt}\\
\cline{2-7}
\end{tabular}
\\
\centering
\begin{tabular}{c|ccccccc|}
\cline{2-8}
$(2\delta_1)-\Omega^{(-)}_{2\delta_1}$ & ${\delta_1-\epsilon}$ & ${\delta_1-\delta_2}$ & ${\delta_1-\delta_3}$ & ${\delta_3+\delta_1}$ & ${\delta_2+\delta_1}$ & ${\delta_1+\epsilon}$ & $\boxed{2\delta_1}$\rule{0pt}{2.6ex}\\
\cline{2-8}
 $(\delta_2+\epsilon)-\Omega^{(-)}_{\delta_2+\epsilon}$ && ${\epsilon-\delta_2}$ & ${\epsilon-\delta_3}$ & ${\delta_3+\epsilon}$ & $\boxed{\delta_2+\epsilon}$ & &  \\
&&& ${\delta_2-\delta_3}$& ${\delta_3+\delta_2}$ & ${2\delta_2}$ & &\\
\cline{2-8}
$(2\delta_3)-\Omega^{(-)}_{2\delta_3}$ &&&& $\boxed{2\delta_3}$ & & &  \\
\cline{2-8}
\end{tabular}
\end{table}

\end{Example}

\subsection{Strongly Orthogonal Subsets of Abelian Ideals} \label{ss:abelianorthog}
The Kostant cascade is constructed by subtracting from maximal roots. It stands to reason that we 
may find similar results when instead adding to minimal roots. We first extend some of Panyushev's 
results from~\cite{Pa3} to GRSs. We do not use Lemma~\ref{L:addstrongorthog}(ii) 
in this paper but include it for completeness.

\begin{Lemma}\label{L:addstrongorthog}
Let $\Delta$ be a GRS. Let $\Delta^+$ be a positive system of $\Delta$, let $A \subseteq \Delta^+$ 
be an abelian ideal of $\Delta^+$, and let $\alpha,\beta\in A$ be strongly orthogonal roots. Then the following holds:
\begin{enumerate}
\item[(i)] if $\alpha+\gamma\in \Delta^+$ for some $\gamma\in \Delta^+\setminus\{0\}$, then $\beta+\gamma\notin \Delta^+$;\
\item[(ii)] if $\alpha-\gamma\in A$ for some $\gamma\in \Delta^+\setminus\{0\}$, then $\beta-\gamma\notin A$.
\end{enumerate}
\end{Lemma}

\begin{proof}
(i) Suppose for a contradiction $\alpha+\gamma\in \Delta^+$ and $\beta+\gamma\in \Delta^+$ for some $\gamma\in \Delta^+\setminus\{0\}$.

Suppose either $(\alpha,\gamma)<0$ or $(\beta,\gamma)<0$. Without loss of generality l
et $(\alpha,\gamma)<0$. Then $(\beta+\gamma,\alpha)<0$ as $(\alpha,\beta)=0$, 
so $(\beta+\gamma)+\alpha\in \Delta^+$. This contradicts $A$ being an abelian ideal.

Suppose $(\alpha,\gamma),(\beta,\gamma)\geq 0$. Then $(\alpha+\gamma,\beta+\gamma)>0$, 
so $(\alpha+\gamma)-(\beta+\gamma)=\alpha-\beta\in R$. This contradicts $\alpha$ and $\beta$ being strongly orthogonal.

Both cases lead to a contradiction, and hence $\beta+\gamma\notin \Delta^+$.

(ii) Let $\alpha-\gamma\in A$ for some $\gamma\in \Delta^+\setminus\{0\}$ and assume for a contradiction that 
$\beta-\gamma\in A$. We claim that $\alpha-\gamma$ and $\beta-\gamma$ are strongly orthogonal: if 
$(\alpha-\gamma)+(\beta-\gamma)\in \Delta^+$, then this contradicts $A$ being an abelian ideal; if 
$(\alpha-\gamma)-(\beta-\gamma)=\alpha-\beta\in R$, then this contradicts $\alpha$ and $\beta$ 
being strongly orthogonal. The result then follows from (i).
\end{proof}
Let $S \subseteq \Delta^+$. Similarly to $\Max(\Delta^+)$, define 
\begin{equation*}
\Min(S)\coloneqq\{\alpha\in S:\alpha\nsucc\beta\text{ for all }\beta\in S\}.
\end{equation*}
Note $\Min(\Delta^+)$ is the base of $\Delta$ associated with $\Delta^+$, which is not a strongly orthogonal subset in 
general. The below proposition shows that $\Min(A)$ is a strongly orthogonal subset when $A$ is a subset of an abelian ideal of $\Delta^+$.

\begin{Proposition} \label{P:MinStrong}
Let $\Delta$ be a GRS. Let $\Delta^+$ be a positive system $\Delta$, let $A \subseteq \Delta^+$ be an 
abelian ideal of $\Delta^+$, and let $S \subseteq  A$. Then $\Min(S)$ is a strongly orthogonal subset of $\Delta$.
\end{Proposition}

\begin{proof}
Assume for a contradiction that there exists distinct $\alpha,\beta\in \Min(S)$ that are not 
strongly orthogonal, so $\alpha+\beta\in \Delta$ or $\alpha-\beta\in \Delta$: if $\alpha+\beta\in \Delta$, 
then $\alpha+\beta\in A$, which contradicts $A$ being abelian; if $\alpha-\beta\in \Delta$, then 
without loss of generality $\alpha-\beta\in \Delta^+$, so $\beta\prec\alpha$, which contradicts 
$\alpha\in\Min(S)$. Both cases lead to a contradiction, so $\Min(S)$ is a strongly orthogonal subset of $\Delta$. 
\end{proof}

We now construct a comparable process to the Kostant cascade for GRSs where we add instead 
of subtract roots. We then provide an example of this.
\begin{Definition}
Let $\Delta$ be a GRS, let $\Delta^+$ be a positive system of $\Delta$, and let $A \subseteq \Delta^+$ be a 
non-empty abelian ideal of $\Delta^+$. Let $\Delta^{(+)}_{\alpha}\coloneqq\{\gamma\in \Delta^+:\alpha+\gamma\in \Delta^+\}$ 
for all $\alpha\in\Delta^+$. Let
\begin{align*}
A_0=A,&&\text{and }&&A_{k+1}=A_k\setminus\bigcup_{\alpha\in \Min(A_k)}(\alpha+\Delta^{(+)}_{\alpha})
\end{align*}
for all $k\in\Z_{\geq 0}$. Let $N\in\Z_{\geq 0}$ be such that $A_{N+1}=\varnothing \ne A_N$. Then define
\begin{equation*}
\mathcal{A}(A)\coloneqq\bigsqcup_{k=0}^{N}\Min(A_k).
\end{equation*}
\end{Definition}

\begin{Example}
We continue with $\Delta^+$ from Example~\ref{E:KCA4}. Consider the abelian 
ideal $A=A_0=\{34,24,14,35,25,15\}$ of $\Delta^+$. We construct $\mathcal{A}(A)$ diagrammatically.

\bigskip

\begin{center}
\includegraphics[width=0.975\linewidth]{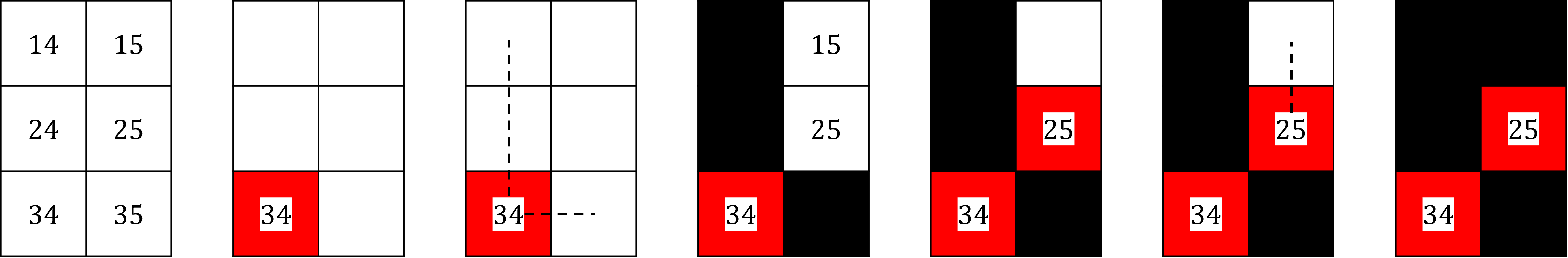}
\end{center}

We find Min$(A_0)=\{34\}$ and $34+\Delta^{(+)}_{34}=\{34,24,14,35\}$, so $A_1=\{25,15\}$. Thus 
$\text{Min}(A_1)=\{25\}$ and $25+\Delta^{(+)}_{25}=\{25,15\}$, so we terminate at $N=1$ with $\mathcal{A}(A)=\{34,25\}$.
\end{Example}

We now prove a similar result to Theorem~\ref{T:KCstrongorth} for abelian ideals.

\begin{Theorem}\label{T:Astrongorth}
Let $\Delta$ be a GRS. Let $\Delta^+$ be a positive system of $\Delta$ and let $A \subseteq \Delta^+$ 
be a non-empty abelian ideal of $\Delta^+$. Then the following holds:
\begin{enumerate}
\item[(i)] $\mathcal{A}(A)$ is a strongly orthogonal subset of $\Delta$;
\item[(ii)] $A=\bigcup_{\alpha\in\mathcal{A}(A)}(\alpha+\Delta^{(+)}_{\alpha})$;\
\item[(iii)] $\Delta^{(+)}_{\alpha}\cap\Delta^{(+)}_{\beta}=\{0\}$ for all $\alpha,\beta\in\mathcal{A}(A)$.
\end{enumerate}
\end{Theorem}

\begin{proof}
The $A=\Min(A)$ case is trivial so we assume that $A\ne\Min(A)$.

\noindent
(i)Assume for a contradiction that there exist distinct 
$\alpha\in \Min(A_n)$ and $\beta\in \Min(A_m)$ that are not strongly orthogonal, for some $n \le m$. At follows from 
Proposition~\ref{P:MinStrong} that 
 $n \ne m$. Thus $n<m$, and $\beta\in A_n$. There are two cases to check:
if $\alpha-\beta\in \Delta^+$, then $\beta\prec\alpha$, which contradicts $\alpha\in \Min(A_n)$; if $\beta-\alpha\in \Delta^+$, then 
$\beta\in(\alpha+\Delta^{(+)}_{\alpha})$, which contradicts $\beta\in A_m$. Hence, $\mathcal{A}(A)$ is a strongly orthogonal subset 
of $\Delta$.

\noindent
(ii)  This follows directly from the definition of $N$.

\noindent
(iii)This follows from Lemma~\ref{L:addstrongorthog}(i) and (i) or the current theorem.
\end{proof}

\section{Evaluating The Index}\label{s:evalindex}

We are now in a position to prove Theorem~\ref{T:index0}, Corollary~\ref{C:indn}, and Theorem~\ref{T:inda}. We first adopt the following assumptions and notations that will be used for the rest of this paper.
\begin{itemize}
\item Let $\g$ be $\gl(m|n)$ or a basic classical Lie superalgebra excluding $\psl(2|2)$. Let $\b=\h\oplus\n$ 
be a Borel subalgebra of $\g$ for $\h$ a Cartan subalgebra of $\g$.
\item Let $\Delta \subseteq \h^*$ be the GRS associated with the the root system of $\g$, and let $\Delta^+$ be the positive system associated with $\b$.
\item Let $\cK(\b)\coloneqq\cK(\Delta^+)$, and let $\mathcal{A}(\a)\coloneqq\mathcal{A}(\Delta^+_{\a})$ for an abelian ideal $\a \subseteq \n$ of $\b$. 
\item Let $0\ne e_{\alpha}\in\g_{\alpha}$ for all $\alpha\in \Delta\setminus\{0\}$ be root space basis vectors.
We allow ourselves to use the notation $e_{\gamma}$ for $\gamma\in(\h^* \setminus \Delta) \cup \{0\}$ with the 
convention that $e_{\gamma}=0$. 
Let $\lambda_{\alpha,\beta}\in\C$ be such 
that $[e_{\alpha},e_{\beta}]=\lambda_{\alpha,\beta}e_{\alpha+\beta}$ for all $\alpha,\beta\in\Delta \setminus \{0\}$ with 
$\alpha \ne -\beta$. Note $\lambda_{\alpha,\beta}\ne 0$ if $\alpha,\beta\ne 0$ and $\alpha+\beta\in\Delta$. 
Let $e_S \coloneqq \sum_{\alpha \in S} e_{\alpha}$ for $S \subseteq \Delta$.
\item Let $h_{\theta}\coloneqq[e_{\theta},e_{-\theta}]$ for all $\theta\in \cK(\b)$, and let $\h_{\cK}\coloneqq\C\{h_{\theta}:\theta\in\cK(\b)\}$.
\end{itemize}

We first prove Theorem~\ref{T:index0}.

\begin{proof}[Proof of Theorem~\ref{T:index0}]
Recall that $\b^*\cong\g/\n$ as a $\b$-module and that we use the notation $[\b,x]_{\b} \coloneqq ([\b,x]+\n)/\n \subseteq \g/\n$ 
for $x \in \g$. 
Let $\alpha \in \Delta^+ \setminus \cK(\b)$. By Theorem~\ref{T:KCstrongorth}
there is a unique $\theta \in \cK(\b)$ such that $\alpha \in \Delta^{(-)}{}_{\theta}$.
Then we have 
\begin{equation*}
[e_{\theta-\alpha},e_{-\cK(\b)}]+\n=\sum_{\beta\in\cK(\b)}[e_{\theta-\alpha},e_{-\beta}]+\n=\lambda_{\alpha,-\theta} e_{-\alpha}+\n.
\end{equation*}
It follows that $e_{-\alpha}+\n \in [\b,e_{-\cK(\b)}]_{\b}$

Observe that $\{h_{\theta}:\theta\in\cK(\b)\}$ is linearly independent as $\cK(\b)$ is orthogonal. Thus for all $\theta\in\cK(\b)$ there exists $h\in\h_{\cK}$ such that $e_{-\theta}+\n = [h,e_{-\cK(\b)}]+\n. \in [\b,e_{-\cK(\b)}]_{\b}$. 

Since $\cK(\b)$ is a strongly orthogonal subset, we have
$[e_{\theta},e_{-\phi}]=0$ for all distinct $\theta,\phi\in\cK(\b)$. 
Therefore, for $\theta \in \cK$ we have that 
$h_{\theta}+\n = [e_{\theta},e_{-\cK(\b)}]+\n \in [\b,e_{-\cK(\b)}]_{\b}$. 

Combining these observations, we deduce that
\begin{equation}\label{e:-cKspan}
[\b,e_{-\cK(\b)}]_{\b} \supseteq \left.\left(\h_{\cK}\oplus\bigoplus_{\alpha\in\Delta^+}\g_{-\alpha}+\n\right)\right/\n. 
\end{equation}
Furthermore it is a straightforward calculation to see that we have the reverse inclusion and thus equality.
Hence, $\dim([\b,e_{-\cK(\b)}]_{\b})=\dim(\n)+\vert\cK(\b)\vert$. Therefore,
\begin{align*}
\ind(\b)&\leq\dim(\b)-\dim([\b,e_{-\cK(\b)}]_{\b})
\\
&=\dim(\b)-(\dim(\n)+\vert\cK(\b)\vert)=\rank(\g)-\vert\cK(\b)\vert.\qedhere
\end{align*}
\end{proof}

The proof of Theorem~\ref{T:index0} leads to the short proof of Corollary~\ref{C:indn} given below.

\begin{proof}[Proof of Corollary~\ref{C:indn}]
Recall $\n^*\cong\g/\b$ as a $\b$-module and that we use the notation $[\b,x]_{\n} \coloneqq ([\b,x]+\b)/\b \subseteq \g/\b$ 
for $x \in \g$. As $\b\supseteq\n$, it follows from~\eqref{e:-cKspan} that
\begin{equation*}
[\b,e_{-\cK(\b)}]_{\n}=\left.\left(\h_{\cK}\oplus\bigoplus_{\alpha\in\Delta^+}\g_{-\alpha}+\b\right)\right/\b=\left.\left(\bigoplus_{\alpha\in\Delta^+}\g_{-\alpha}+\b\right)\right/\b=\g/\b.
\end{equation*}
Thus $\ind(\b,\n)=0$.
This further implies $\ind(\b,\i)=0$ for all ideals $\i \subseteq \n$ of $\b$ as $\i^{\perp} \subseteq (\n)^{\perp}$.
\end{proof}

We lastly consider the index of the dual of abelian $\b$-modules.
\begin{proof}[Proof of Theorem~\ref{T:inda}]
Let $\a \subseteq \n$ be an abelian ideal of $\b$. Let $\beta\in \Delta^+(\a) \setminus \mathcal{A}(\a)$.
By Theorem~\ref{T:Astrongorth}(ii) there exists $\alpha \in \mathcal{A}(\a)$ such that $\beta - \alpha \in \Delta^+$.
 As $\mathcal{A}(\a) \subseteq \Delta^+_{\a}$ is a strongly orthogonal subset contained in an abelian ideal, it follows 
 from Lemma~\ref{L:addstrongorthog}(i) that
\begin{equation*}
[e_{\beta-\alpha},e_{\mathcal{A}(\a)}]=\sum_{\gamma\in\mathcal{A}(\a)
}[e_{\beta-\alpha},e_{\gamma}]=\lambda_{\beta-\alpha,\alpha} e_{\beta}.
\end{equation*}
It follows that $e_\beta \in [\b,e_{\mathcal{A}(\a)}]$.

Observe $\mathcal{A}(\a)$ is linearly independent as $\mathcal{A}(\a)$ is orthogonal. Thus for all $\gamma\in \mathcal{A}(\a)$ there exists $h\in \h$ such that $[h,e_{\mathcal{A}(\a)}]=e_{\gamma}$. 

Combining these observations, 
we obtain that $[\b,e_{\mathcal{A}(\a)}]=\a$. Therefore,
$\ind(\b,\a^*)=0$.
\end{proof}

\end{document}